\documentclass[a4paper,12pt]{article}
\usepackage[utf8]{inputenc}

\title{Second homotopy and invariant geometry of flag manifolds}

\author{Lino Grama\footnote{Departamento de Matemática, UNICAMP - Universidade Estadual de Campinas - SP, lino@ime.unicamp.br} \, and \, 
Lucas Seco\footnote{Departamento de Matemática, UnB - Universidade Federal de Brasília - DF, lseco@unb.br}
}

\usepackage{amsthm}
\usepackage{amsfonts}
\usepackage{amssymb}
\usepackage{amsmath}

\usepackage{tikz-cd}
\tikzset{
         isom/.style={every to/.append style={edge node={node [sloped, allow upside down, auto=false]{$^\sim$}}}}}

\usepackage{fancyhdr}
\usepackage{lastpage}
\usepackage{enumerate}
\usepackage[totpages]{zref}
\usepackage{framed}

\usepackage{graphicx}
\usepackage{color}

\newcommand{\sen}{\,{\rm sin}}
\newcommand{\tg}{\,{\rm tan}}

\newcommand{\Ad}{\mathrm{Ad}}

\newcommand{\g}{\mathfrak{g}}

\newcommand{\h}{\mathfrak{h}}

\renewcommand{\u}{\mathfrak{u}}
\newcommand{\su}{\mathfrak{su}}

\renewcommand{\t}{\mathfrak{t}}
\newcommand{\F}{\mathbb F}
\newcommand{\T}{\Theta}

\newcommand{\R}{\mathbb{R}}
\newcommand{\C}{\mathbb{C}}
\renewcommand{\H}{\mathbb{H}}
\newcommand{\Z}{\mathbb{Z}}

\newcommand{\ad}{\mathrm{ad}}

\newcommand{\ds}{\displaystyle}

\newcommand{\m}{\mathfrak{m}}

\newcommand{\ov}[1]{{\overline{#1}}}
\newcommand{\wt}[1]{{\widetilde{#1}}}

\usepackage{calrsfs}
\usepackage[all,cmtip]{xy}
\usepackage{mathtools}

\newtheorem{theorem}{Theorem}[section]
\newtheorem{lemma}[theorem]{Lemma}
\newtheorem{proposition}[theorem]{Proposition}
\newtheorem{corollary}[theorem]{Corollary}
\newtheorem{example}[theorem]{Example}
\newtheorem{remark}[theorem]{Remark}
\theoremstyle{definition}
\newtheorem{definition}[theorem]{Definition}

\renewcommand{\u}{\mathfrak{u}}

\renewcommand{\t}{\mathfrak{t}}
\renewcommand{\m}{\mathfrak{m}}

\newcommand{\ii}{{\rm \bf i}}
\renewcommand{\j}{{\rm \bf j}}
\renewcommand{\k}{{\rm \bf k}}

\def\prod#1{\langle{#1}\rangle}

\DeclareFontFamily{U}{mathx}{\hyphenchar\font45}
\DeclareFontShape{U}{mathx}{m}{n}{
      <5> <6> <7> <8> <9> <10>
      <10.95> <12> <14.4> <17.28> <20.74> <24.88>
      mathx10
      }{}
\DeclareSymbolFont{mathx}{U}{mathx}{m}{n}
\DeclareFontSubstitution{U}{mathx}{m}{n}
\DeclareMathAccent{\widecheck}{0}{mathx}{"71}
\DeclareMathAccent{\wideparen}{0}{mathx}{"75}

\begin{document}

\maketitle

\begin{abstract}
We use the Hopf fibration to explicitly compute generators of the second homotopy group of the flag manifolds of a compact Lie group. We show that these $2$-spheres  have nice geometrical properties such as being totally geodesic surfaces with respect to any invariant metric on the flag manifold. We characterize when the generators with the same invariant geometry are in  the same homotopy class. This is done by exploring the action of Weyl group on the irreducible components of isotropy representation of the flag manifold.
\end{abstract}

\medskip


\section*{Introduction}
We consider (generalized)  flag manifolds of a compact Lie group $U$, that is, homogeneous manifolds $\mathbb{F}_\T=U/U_\T$, where the isotropy $U_\T$ is a connected subgroup with maximal rank.  These omnipresent manifolds have being studied from several points of view: symplectic geometry, algebraic geometry, riemannian geometry and combinatorics (see for instance \cite{besse}, \cite{biron}, \cite{GGSM}, \cite{lonardo-jordan}). Moreover, several geometric objects on flag manifolds can be described in a very explicit way in terms of Lie theory and representation theory. 

In this work, we describe explicitly topological phenomena on flag manifolds via invariant geometry and geometry of root systems.
One of our motivation is the work of Dur\'an-Mendoza-Rigas \cite{duran-rigas}, where the authors give a geometric description of Blakers-Massey element, the generator of the homotopy group $\pi_6(S^3)$, by considering the principal bundle $S^3 \cdots {\rm Sp}(2) \to S^7$ and explicitly describing the boundary homomorphism  
$\pi_7(S^7) \stackrel{\partial}{\to} \pi_6(S^3)$.

In the first part of the paper (Section \ref{boundary-sec}) we consider the second homotopy group of the flag manifold $\mathbb{F}_\T$.  It is well known to be generated by 2-spheres given by the Schubert cells associated to reflections of the {\em simple} roots outside the simple roots of the isotropy.  In our first result, we consider the
principal bundle $U_\T \cdots U \to \F_\T$ and explicitly describe the boundary homomorphism $\pi_2(\F_\T) \stackrel{\partial}{\to}  \pi_1(U_\T)$ on special 2-spheres $\sigma^\vee_\alpha: S^2 \to \F_\T$ associated to {\em all} the roots outside the roots of the isotropy.  As an application we describe precisely the homotopy class of some 2-spheres by using the combinatory of the root system.

Regarding differential geometry,  in \cite{twistor} Burstall-Rawnsley showed that the generators of $\pi_2(\F_\T)$ are totally geodesic with respect to a specific metric (the so called normal metric) and in particular, it is an harmonic map with respect to that metric. In the second part (Section \ref{equiharmonic-spheres}) of this paper we explore the properties of our generators $\sigma^\vee_\alpha$ in terms of the harmonic map theory and, in our second result, we prove that these generators are totally geodesic with respect to {\em any} invariant metric on $\F_\T$. It follows that $\pi_2(\F_\T)$ is generated by equiharmonic maps, that is, harmonic maps with respect to any invariant metric on $\F_\T$.

A very important ingredient in the study of the geometry of flag manifolds is its isotropy representation and the corresponding isotropy components. It allows us to describe invariant tensors in the flag manifold (e.g. invariant metrics, almost complex structures, symplectic forms and so on). In the third part of the paper (Sections \ref{theta-rigid} and \ref{actionW}) we explore the relation between the invariant geometry of the generators $\sigma^\vee_\alpha$ and its homotopy class. This is done after a careful analysis of the action of  the Weyl group on the irreducible components of the isotropy representation, which is purely a problem on the geometry of the root system. In particular we introduce the concept of $\T$-rigid roots to answer, in our third result, the following question: when are two {\em isometric} generators of $\pi_2$, say $\sigma^\vee_\alpha$ and $\sigma^\vee_\beta$ in the same {\em homotopy class}?
We feel that there are interesting examples of the interplay between invariant geometry and topology to be explored when the invariant geometry and homotopy classes of these spheres do not coincide, for example, the partial flag manifolds of type $B$, $C$, $F$ or $G$, acoording to our result (see Example \ref{exemplog2} for type $G$).

We finish the introduction with two remarks on the generality of our setup.  In the first place, we avoided to take the compact connected group $U$ semisimple and simply connected from the start. This makes the setup more flexible to produce examples and easily adaptable to the context of maximal rank homogeneous spaces $U/L$, that is, where the isotropy $L$ has maximal rank but can be disconnected.  Let $L_0$ be the connected component subgroup of $L$, then $U/L_0$ is a flag manifold of $U$ and the fibration $U/L_0 \to U/L$ is the universal covering of $U/L$ which, thus, induces isomorphism in the higher homotopy groups.  Thus, the results of Section 2 follows through: the 2-spheres $\sigma^\vee_\alpha$ project to generators of $\pi_2(U/L)$, whose images can be 2-spheres or projective planes.  
This can be of use to study the topology of {\em real flag manifolds} of maximal rank.  For example, the real flag manifold of maximal rank ${\rm SO}(4)/{\rm S}({\rm O}(2) \times {\rm O}(2))$ is the Grassmannian of planes of $\R^4$ and its universal cover is ${\rm SO}(4)/({\rm SO}(2) \times {\rm SO}(2))$, the oriented Grassmannian of planes of $\R^4$, the maximal flag manifold of the non-simply connected group ${\rm SO}(4)$. In the second place, we prove the results on the action of the Weyl group on the isotropy  components in the more general context of {\em nonreduced} root systems. Again, this may be of future interest for the study of the invariant geometry of real flag manifolds (see \cite{mauro-luiz}).

\section{Preliminaries}

We recall some results and constructions on the fundamental group of compact Lie groups, since we will use similar arguments in what follows.  See Helgason \cite{helgason}
or Hilgert and Neeb \cite{neeb} for details and proofs.
Let a compact connected Lie group $U$ have Lie algebra $\u$ and exponential map $\exp: \u \to U$.
Let $T$ have Lie algebra $\t$ and {\em lattice} given by
$$
\Gamma = \{ \gamma \in \t: \, \exp( 2\pi \gamma ) = 1 \}
$$
which is canonically isomorphic to $\pi_1(T)$ by the map $\Gamma \to \pi_1(T)$ that takes $\gamma \in \Gamma$ to the based homotopy class of the loop $S^1 \to T$, $e^{\theta \ii}\mapsto \exp(\theta \gamma)$. 
Composing with the map induced by the inclusion $T \subset U$ gives a surjective map $\Gamma \to \pi_1(U)$ whose kernel is given by the following construction (see Theorem 12.4.14 p.494 of \cite{neeb}).

\subsubsection*{Coroot vectors}

We have that $\u$ is the compact real form of the complex reductive Lie algebra $\g = \u^\C$.  The adjoint representation of the Cartan subalgebra $\h = \t^\C$ splits as the root space decomposition
$
\g = \h \oplus \sum_{\alpha \in \Pi} \g_\alpha
$
with root space
$$
\g_\alpha = \{ X \in \g:\, \ad(H) X =  \alpha(H) X, \, \forall H \in \h  \}
$$
where $\Pi \subset \h^*$ is the root system.  We have that $\dim \g_\alpha = 1$ and 
that each root $\alpha$ is imaginary valued in $\t$ so that $\alpha \in \ii \t^*$.
To each root $\alpha$ there corresponds the unique {\em coroot vector} $H^\vee_\alpha \in \h$ such that
$$
H^\vee_\alpha \in [\g_\alpha, \, \g_{-\alpha}]\quad\text{and}\quad 
\alpha(H^\vee_\alpha) = 2
$$


Now, let $X \mapsto \ov{X}$ denote conjugation in $\g$ w.r.t.\ $\u$, it is an automorphism of $\g$.  We have that $\ov{\h} = \h$ and, for a root $\alpha$, $\ov{\alpha(H)} = -\alpha( \ov{H} )$ for $H \in \h$, so that $\ov{\g_\alpha} = \g_{-\alpha}$.  
For $X_{\alpha} \in \g_{\alpha}$, we have that both
$$
A_\alpha = \frac{1}{2}(X_\alpha - \ov{X_\alpha}) 
\qquad
S_\alpha = \frac{1}{2\ii}(X_\alpha + \ov{X_\alpha})
$$
belong to $\u$ and are such that $[H, A_\alpha] = \ii \alpha(H) S_\alpha$ and
$[H, S_\alpha] = -\ii \alpha(H) A_\alpha$, for $H \in \t$.
Consider the real root space 
$$
\u_\alpha = {\rm span}_\R \{ A_\alpha, S_{\alpha} \}
$$
Let $\Pi^+$ be a choice of positive roots, then $\u$ splits as the real root space decomposition
$$
\u = \t \oplus \sum_{\alpha \in \Pi^+} \u_\alpha
$$

Choosing $X_{\alpha} \in \g_\alpha$ such that $[X_\alpha, \ov{X_\alpha}] = H^\vee_\alpha$ we have that $\ii H^\vee_\alpha \in \t$ and that  $[S_\alpha, A_\alpha] = \ii H^\vee_\alpha/ 2$. Furthermore
$[\ii H^\vee_\alpha, A_\alpha] = - 2 S_\alpha$
and
$[\ii H^\vee_\alpha, S_\alpha] = 2 A_\alpha$ 
so that
$$ 
\u(\alpha) = {\rm span}_\R \{ S_\alpha, \, \ii H^\vee_\alpha, \, A_{\alpha} \}
$$
is a subalgebra of $\u$ isomorphic with $\su(2)$ by the explicit isomorphism 
$$
\frac{1}{2\ii}
\begin{pmatrix}
0 & \phantom{-}1 \\
1 & \phantom{-}0
\end{pmatrix}
\mapsto
S_\alpha
\qquad
\begin{pmatrix}
\ii & \phantom{-}0 \\
0 & -\ii
\end{pmatrix}
\mapsto
\ii H^\vee_\alpha
\qquad
\frac{1}{2}
\begin{pmatrix}
\phantom{-}0 & \phantom{-}1 \\
-1 & \phantom{-}0
\end{pmatrix}
\mapsto
A_\alpha
$$
Let $U(\alpha)$ be the connected subgroup of $U$ with Lie algebra $\u(\alpha)$.  Since $\u(\alpha)$ is semisimple, if follows that $U(\alpha)$ is compact.  This  isomorphism $ \psi: \su(2) \to \u(\alpha)$ can be uniquely integrated to a group epimorphism $ \Psi: SU(2) \to U(\alpha)$, since $SU(2)$ is simply connected.
It maps
$$
\Psi
\begin{pmatrix}
e^{\theta \ii}&    \\
 & e^{-\theta \ii}\\
\end{pmatrix}
=
\exp( \theta \ii H^\vee_\alpha)
$$
so that $\exp(2 \pi \ii H^\vee_\alpha) = 1$ and thus $\ii H^\vee_\alpha \in \Gamma$.  

The $\Z$-span of the coroot vectors $\ii H^\vee_\alpha$, $\alpha \in \Pi$,  
gives the {\em coroot group} $\Gamma^\vee$, a subgroup of $\Gamma$.  
The map $\Gamma \to \pi_1(U)$ is then an epimorphism with kernel $\Gamma^\vee$ so that it induces a natural isomorphism 
\begin{equation}
\label{eq:isompi1}
\Gamma / \Gamma^\vee \to \pi_1(U) 
\end{equation}
This can interpreted geometrically as: the loops in $T$ generate the loops in $U$, but the loops in $U$ coming from a coroot $\ii H^\vee_\alpha$ shrinks to the identity since it comes from the ``equator'' of the simply connected $SU(2)$ which is then immersed in $U(\alpha)$.

\subsubsection*{Dual roots and Weyl group}

Choose an $\Ad(U)$ invariant inner product $\prod{\cdot,\cdot}$ in $\u$ and extend it to an Hermitian product in $\g$ where $\u$ and $\ii \u$ are orthogonal.  
Restricting it to $\h$ gives a linear isomorphism $\h^* \to \h$ that associates to each functional $\phi \in \h^*$ a vector $H_\phi \in \h$ such that $\phi(H) = \prod{H_\phi, H}$ for all $H \in \h$.   This isomorphism furnishes in $\h^*$ with  an Hermitian product and its restriction to the real subspace spanned by the roots is an inner product.  To a root $\alpha$, since $\alpha \in \ii \t^*$, there corresponds an $H_\alpha \in \ii \t$ that is proportional to the coroot vector $H^\vee_\alpha$, with proportionality given by
$$
H^\vee_\alpha = \frac{2 H_\alpha}{\prod{\alpha, \alpha}}
$$
The dual root system $\Pi^\vee$ of $\Pi$ is given by the dual roots
$$
\alpha^\vee = \frac{2 \alpha}{\prod{\alpha, \alpha}}
$$
It follows that, under this isomorphism, the coroot vectors correspond to dual roots, in other words, $H^\vee_\alpha = H_{\alpha^\vee}$.  

The reflection $r_\alpha$ of $\t$ around the hyperplane $\alpha = 0$ is given by $r_\alpha(H) = H - \alpha(H) \alpha^\vee$. 
Let $W = N(T)/T$ be the Weyl group of $U$.  Fix a choice $\Sigma$ of simple roots.  We see an element $w \in W$ either as an element of $N(T)$ acting by the adjoint action on $\t$, or as an element of the isometry subgroup of $\t$ generated by simple reflections $r_\alpha$, $\alpha \in \Sigma$.
Let $H \in \t$ and $u \in U$ be such that $\Ad(u) H \in \t$, then $\Ad(u)H = w H$ for some $w \in W$ (see Proposition VII.2.2 p.285 of \cite{helgason}).
We have that $W$ also acts on the roots $\Pi$ by the coadjoint action $w^*\alpha(H)=\alpha(w^{-1}H)$.

\section{Boundary homomorphism and $\pi_2$}
\label{boundary-sec}
In this section we use the boundary homomorphism to investigate the second homotopy group of flag manifolds.

\subsection{Hopf fibration}
\label{sec:hopf}
Consider the Hopf fibration $U(1) \cdots SU(2) \stackrel{h\,}{\to} \C P^1$ over the  Riemann sphere $\C P^1$, given by
$$
h
\begin{pmatrix}
z & -\ov{w} \\
w & \phantom{-}\ov{z}
\end{pmatrix}
= z/w
\qquad
\text{with fiber}
\qquad
U(1) = \left\{
\begin{pmatrix}
\lambda & \\
                 & \ov{\lambda}
\end{pmatrix} 
:\, 
\begin{array}{r}
\lambda \in \C \\
|\lambda|=1
\end{array}
\right\}
$$
over the basepoint $\infty =z/0$, where $z, w \in \C$ are such that $|z|^2 + |w|^2 = 1$.
We explicitly compute the boundary isomorphism of the homotopy exact sequence 
\begin{equation}
\label{eq:seqexata}
0= \pi_2(SU(2)) \to \pi_2(\C P^1) \stackrel{\partial}{\to} \pi_1(U(1)) \to \pi_1(SU(2)) = 0
\end{equation}

Consider the continuous map 
$$
f: D \subset \C \to \C P^1 \qquad z = r e^{\theta \ii} \mapsto \tg(\pi r/2) e^{\theta \ii}
$$
from the disk $D$ of radius $r \leq 1$, $\theta \in [0,2\pi]$, to the Riemann sphere. 
It is continuous, maps the origin to the origin, the disk interior $r < 1$ homeomorphically into the sphere minus $\infty$ and maps the whole boundary $r = 1$ onto the basepoint $\infty$.  It follows that it induces the generator of $\pi_2(\C P^1)$. 
Consider a continuous lift of $f$ to $SU(2)$ given by
\begin{equation}
F: D  \to SU(2) \qquad z = r e^{\theta \ii} \mapsto 
\left(
\begin{array}{ll}
\sen(\pi r/2) e^{\theta \ii}& -\cos(\pi r/2) \\
\cos(\pi r/2)                  & \phantom{-}\sen(\pi r/2) e^{-\theta \ii}
\end{array}
\right)
\end{equation}
where, clearly, $h \circ F = f$.
At 
the boundary $r = 1$ of $D$ we have
$$
F|_{\partial D}:\,
e^{\theta \ii}
\mapsto
\begin{pmatrix}
e^{\theta \ii}& \\
                 & e^{-\theta \ii}
\end{pmatrix} 
$$
Thus, in homotopy we have
$$
\partial(1) = \partial( [f] ) = [F|_{\partial D}] = 
\left[ e^{\theta \ii}\mapsto 
\begin{psmallmatrix}
e^{\theta \ii}& \\
   & e^{-\theta \ii}
\end{psmallmatrix} 
\right] = 1
$$
where the leftmost $1$ is the generator of $\pi_2(\C P^1)$ on the base and the rightmost $1$ is the generator of $\pi_1(U(1))$ on the fiber.

Consider now the quotient Hopf fibration 
$
U(1)/\{ \pm 1 \} \cdots SU(2)/\{ \pm 1 \} \stackrel{h\,}{\to} \C P^1
$
given by the same map.  It is actually the frame bundle fibration
$$
SO(2) \cdots SO(3) \stackrel{g}{\to} S^2
$$
over the Euclidean unit sphere $S^2 \subset \R^3$, given by 
$$
g(*) = \text{last column of }* 
\quad\quad \text{with fiber} \quad\quad
\begin{pmatrix}
SO(2)  & \\
& 1 \\
\end{pmatrix}
$$
over the basepoint $e_3 = (0,0,1)$, were $SO(3)$ acts canonically in $\R^3$ and we identify the fiber with $SO(2)$.  In fact, let
$\R^3$ be the imaginary quaternions spanned by $\ii, \j, \k$, so that $e_3 = \k$.  Since $SU(2)$ can be identified with the unit quaternions $q = z + \j w$, its Lie algebra is then this $\R^3$. The 2-covering epimorphism $SU(2) \to SO(3)$ is then the adjoint representation $\Ad$, the action of the unit quaternions on $\R^3$ by conjugation, and has kernel $\{\pm 1\}$.  Note that $h(q) = z/w$ and that $g(  \Ad(q) ) =  q \k q^{-1}$. The following diagram commutes
\begin{equation}
\label{eq:SU2-SO3}
\begin{array}{ccccc}
U(1) & \cdots & SU(2) & \stackrel{h}{\to} & \C P^1 \\
\text{\tiny 1:2} \downarrow \phantom{\Ad} &  & \text{\tiny 1:2} \downarrow \Ad  & &  
\text{\tiny 1:1} \downarrow \phi  \\
SO(2) & \cdots & SO(3) & \stackrel{g}{\to} & S^2 
\end{array}
\end{equation}
where $\phi$ is the diffeomorphism defined by $\phi( h(q) ) =  q \k q^{-1}$, $q$ a unit quaternion in $SU(2)$.  Quotienting the groups in first row of the diagram above we get that the frame bundle fibration is isomorphic to the quotient Hopf fibration.  Computing its boundary $\ov{\partial}: \pi_2(S^2)  \to \pi_1(SO(2))$
we have, by naturality, that the following diagram commutes
$$
\begin{array}{ccc}
\pi_2(\C P^1) & \stackrel{\partial}{\to} &  \pi_1(U(1)) \\
\text{\tiny 1:1} \downarrow \phi & &   \text{\tiny 1:2} \downarrow \wt{\pi}   \\
\pi_2(S^2)  & \stackrel{\ov{\partial}}{\to} & \pi_1(SO(2))
\end{array}
$$
where we denoted the induced maps in homotopy again by their names.  It follows that
$$
\ov{\partial}(1) = 2
$$
where the leftmost $1$ is the generator of $\pi_2(S^2)$ on the base and the rightmost $2$ is twice the generator of $\pi_1(SO(2))$ on the fiber.
From now on we topologically identify 
$$
\C P^1 = S^2
$$

\begin{remark}
\label{remark:quat}
Consider the 1 by 1 unitary symplectic group ${\rm Sp}(1)$ of unitary quaternions, isomorphic to $SU(2)$, the 2 by 2 unitary symplectic group of quaternionic matrices
$$
{\rm Sp}(2) = \left\{
\begin{pmatrix}
p & r \\
q & s
\end{pmatrix}:\, p, q, r, s \in \H, \quad 
\begin{array}{ll}
|p|^2 + |q|^2 = 1, & |r|^2 + |s|^2 = 1 \\
\ov{p}r   + \ov{q}s = 0, & p\ov{q}   + r\ov{s} = 0 
\end{array}
\right\}
$$
and the Hopf fibration ${\rm Sp}(1) \times {\rm Sp}(1) \cdots {\rm Sp}(2) 
\stackrel{h}{\to} \H P^1$ over the quaternionic projective line $\H P^1 = S^4$, given by
$$
h
\begin{pmatrix}
p & r \\
q & s
\end{pmatrix}
= p/q
\qquad
\text{with fiber}
\qquad
{\rm Sp}(1) \times {\rm Sp}(1) = \left\{
\begin{pmatrix}
\lambda & \\
                 & \mu
\end{pmatrix} 
\right\}
$$
over the basepoint $\infty =p/0$, where $\lambda, \mu \in \H$ are such that $|\lambda| = |\mu| = 1$.
The same argument as before can be used to compute the boundary map $\pi_4(\H P^1) \stackrel{\partial}{\to} \pi_3(S^3)$.  Indeed, consider the continuous map 
$$ 
f: D \subset \H \to \H P^1  \qquad  q  \mapsto \tg(\pi |q|/2) \frac{q}{|q|}
$$
from the ball $D$ of radius $\leq 1$.  It is continuous, maps the origin to the origin, the ball interior $r < 1$ homeomorphically into the sphere minus $\infty$ and maps the whole boundary $r = 1$ onto the basepoint $\infty$.  It follows that it induces the generator of $\pi_4(\H P^1)$.  Consider a continuous lift of $f$ to ${\rm Sp}(2)$ given by
\begin{equation}
F: D  \to {\rm Sp}(2) \qquad q  \mapsto 
\left(
\begin{array}{ll}
\ds \sen(\pi |q|/2)\frac{q}{|q|} & -\cos(\pi |q|/2) \\
\cos(\pi |q|/2)                  & \ds \phantom{-}\sen(\pi |q|/2) \frac{\ov{q}}{|q|}
\end{array}
\right)
\end{equation}
where, clearly, $h \circ F = f$. 
Thus, in homotopy we have 
$$
\partial(1) = \partial( [f] ) = [F|_{\partial D}] = 
\left[ q \mapsto 
\begin{psmallmatrix}
q & \\
   & \ov{q}
\end{psmallmatrix} 
\right] = 
(1, -1)
$$
where $q$ stands for unitary quaternions $|q|=1$, the leftmost $1$ is the generator of $\pi_4(\H P^1)$ on the base and the rightmost $1$ is the generator of $\pi_3( {\rm Sp}(1) ) = \Z$ on each factor of the fiber.

We also have the quotient Hopf fibration 
$$
{\rm Sp}(1) \times {\rm Sp}(1)/\{ \pm 1 \}  \cdots
{\rm Sp}(2)/\{ \pm 1 \} \stackrel{h\,}{\to} \H P^1
$$
given by the same map, which is actually the frame bundle fibration
$$
{\rm SO}(4) \cdots {\rm SO}(5) \stackrel{g}{\to} S^4
$$
over the Euclidean unit sphere $S^4 \subset \R^5$.
This is because ${\rm Sp}(2)$ is the universal cover of ${\rm SO}(5)$ and ${\rm Sp}(1) \times {\rm Sp}(1)$ is the universal cover of ${\rm SO}(4)$, both with the same kernel $\{\pm 1\}$ (see Examples 6.50 (b) and (c), p.208 of \cite{poor}).
Computing the boundary $\ov{\partial}: \pi_4(S^4)  \to \pi_3({\rm SO}(4))$
we have, by naturality, that the induced diagram in homotopy commutes
$$
\begin{array}{ccc}
\pi_4(\H P^1) & \stackrel{\partial}{\to} &  \pi_3( {\rm Sp}(1) \times {\rm Sp}(1) ) \\
\text{\tiny 1:1} \downarrow & & \text{\tiny 1:1} \downarrow  \\
\pi_4(S^4)  & \stackrel{\ov{\partial}}{\to} & \pi_3({\rm SO}(4))
\end{array}
$$
where the righmost column is an isomorphism, since the universal covering induces isomorphism in higher homotopy groups.  
It follows that
$$
\ov{\partial}(1) = (1,-1)
$$
where the leftmost $1$ is the generator of $\pi_4(S^4)$ on the base and the rightmost 
$1$ is the generator of $\pi_3( {\rm Sp}(1) ) = \Z$ on each factor of the fiber $\pi_3({\rm SO}(4)) \simeq \pi_3({\rm Sp}(1) \times {\rm Sp}(1))$.
\end{remark}

\subsection{Flag fibration}

Let $\F_\T = U/U_\T$ be a flag manifold of a connected compact Lie group $U$, that is, the isotropy $U_\T$ is centralizer of a torus.  It follows that $U_\T$ is connected.
In this section we consider the flag fibration $U_\T \cdots U \to \F_\T$ with basepoint $1$ on the fiber (thus on the total space) and basepoint $o = U_\T$ (trivial coset) on the base.
By reducing to rank one, we compute the {\em boundary map} of the homotopy exact sequence 
\begin{equation}
\label{eq:seqexata}
0 = \pi_2(U) \to \pi_2(\F_\T) \stackrel{\partial}{\to} \pi_1(U_\T) \to \pi_1(U)
\end{equation}
on special elements of $\pi_2(\F_\T)$.  Then we use this to show that these special elements provide a $\Z$-basis for $\pi_2(\F_\T)$.  Note that the second homotopy group of a connected Lie group is zero so that, by exactness, the boundary map is always injective. 

Note that different groups $U$ may yield diffeomorphic quotients $\F_\T$, nonetheless, the boundary homomorphism depends on the chosen groups since it depends on the topology of the whole fibration.

\begin{example}
The only flag manifold of rank 1 is the sphere $S^2$ considered in the previous section.
If the group is simply connected, the flag fibration is the Hopf fibration whose boundary map is surjective.
If the group is not simply connected, the flag fibration is the frame bundle fibration whose boundary map is not surjective.
\end{example}

First we identify the image of the boundary map.   By exactness, it is the kernel of the map $\pi_1(U_\T) \to \pi_1(U)$.  Fix $T$ a maximal torus of $U_\T$, hence of $U$.  It follows that $U$ and $U_\T$ share the same lattice $\Gamma$.  
Denote by $\Pi$ the roots of $\u$, it contains the roots $\Pi_\Theta$ of $\u_\T$, the Lie algebra of $U_\T$.  Denote by $\Gamma^\vee$ the coroot lattice of $\u$, it contains the coroot lattice $\Gamma^\vee_\Theta$ of the isotropy $\u_\T$.  
Fix from now on the natural isomorphisms that follow from (\ref{eq:isompi1})
$$
\Gamma / \Gamma^\vee_\T \to \pi_1(U_\T) \qquad
\Gamma / \Gamma^\vee \to \pi_1(U) 
$$
Note that $\Gamma/\Gamma^\vee$ can have torsion, as in $\pi_1(SO(3)) = \Z_2$.

Denote by $R^\vee$ the dual root group, given by the $\Z$-span of the dual roots $\Pi^\vee$. 
Since the isomorphism $\phi \in \h^* \mapsto \ii H_\phi \in \h$ takes a dual root $\alpha^\vee$ to the coroot vector $\ii H_{\alpha^\vee} = \ii H^\vee_{\alpha}$, it follows that it restricts to a group isomorphism
$$
R^\vee \to \Gamma^\vee
$$ 
Denote by $R^\vee_\T$ the dual root group of the isotropy, given by the $\Z$-span of the dual roots $\Pi^\vee_\T$.  Then the isomorphism above restricts to an isomorphism  $R^\vee_\T \to \Gamma^\vee_\T$ so that we have the natural group isomorphism
\begin{equation}
\label{isom:raizdual-coraiz}
R^\vee/R^\vee_\T  \to \Gamma^\vee/\Gamma^\vee_\T
\end{equation}
Fix a system of simple roots $\T$ for $\Pi_\T$ and extend it to a system of simple roots $\Sigma$ for $\Pi$. This gives a choice of positive roots $\Pi^+$ and then $\Pi^+_\T$.

\begin{proposition}
\label{prop:imagem}
The image of $\partial$ is $\Gamma^\vee/\Gamma^\vee_\T$ and has no torsion. Furthermore, $\partial$ is surjective iff $U$ is simply connected.
\end{proposition}
\begin{proof}
Under the above isomorphims, the map $\pi_1(U_\T) \to \pi_1(U)$ induced by the inclusion $U_\T \subset U$ clearly becomes the natural map $\Gamma / \Gamma^\vee_\T \to \Gamma / \Gamma^\vee$, which has kernel $\Gamma^\vee/\Gamma^\vee_\T$.  

For the torsion statement, 
note that the dual simple roots $\Sigma^\vee$ is a $\Z$-basis for $R^\vee$ and the dual simple roots $\Sigma^\vee_\T$ is a $\Z$-basis for $R^\vee_\T$.  Since this $\Z$-basis of $R^\vee_\T$ is a subset of this $\Z$-basis of $R^\vee$, it follows that 
$R^\vee/R^\vee_\T$ has no torsion.  The result then follows from the isomorphism (\ref{isom:raizdual-coraiz}).

If $\pi_1(U)=0$ then $\partial$ is surjective by the exact sequence (\ref{eq:seqexata}).
Reciprocally, if $\partial$ is surjective then $\Gamma^\vee / \Gamma^\vee_\T = \Gamma / \Gamma^\vee_\T$ so that $\Gamma^\vee = \Gamma$ and then $\pi_1(U) =
\Gamma/\Gamma^\vee = 0$.
\end{proof}
The previous result shows that the image of the boundary homomorphism depends only on the Lie algebra of $U$ and $U_\T$.   We now construct a special element of $\pi_2(\F_\T)$ which gets mapped to each coroot vector in this image.  Let $\alpha \in \Pi$ be a root, not necessarily simple. Consider the compact subgroup $U(\alpha)$ of $U$ 
and epimorphism $\Psi: SU(2) \to U(\alpha)$.
Since
$
\Psi\begin{psmallmatrix}
e^{\theta \ii}&    \\
 & e^{-\theta \ii}\\
\end{psmallmatrix}
=
\exp( \theta \ii H^\vee_\alpha)
$
it follows that $\Psi(U(1)) 
\subset T \subset U_\T$.
We then get the commutative diagram of immersions
\begin{equation}
\label{eq:hopf-em-U}
\begin{array}{rcrcl}
U(1) \,\,\,\, & \cdots & SU(2) &  \stackrel{h}{\to} & S^2 \quad \\
 \downarrow \Psi &  & \downarrow \Psi &  &  \downarrow  \sigma^\vee_\alpha \\
U_\T \,\,\,\, & \cdots & U \quad & \to & \F_\T 
\end{array}
\end{equation}
where in the upper row we have the Hopf bundle and $\Psi$ factors to the immersion of a 2-sphere into $\F_\T$ given by
$$
\sigma^\vee_\alpha: 
S^2 \to \F_\T
\qquad 
z/w
\mapsto \Psi \begin{pmatrix}
z & -\ov{w} \\
w & \ov{z}
\end{pmatrix}
o
$$
whose image is the orbit of $U(\alpha)$ through the basepoint $o$ of $\F_\T$.  
Denote the based homotopy class of this map again by $\sigma^\vee_\alpha \in \pi_2(\F_\T)$.

\begin{proposition}
\label{teo:bordo}
$\partial(  \sigma^\vee_\alpha ) = \ii H^\vee_\alpha \mod \Gamma^\vee_\T$
\end{proposition}
\begin{proof}
The homotopy class of $\sigma^\vee_\alpha$ in $\F_\T$ is represented by $\Psi \circ f: D \to \F_\T$. This has a lift  $\Psi \circ F: D \to U$ such that
$$
\Psi \circ F|_{\partial D} =
\Psi
\begin{pmatrix}
e^{\theta \ii}& \\
                 & e^{-\theta \ii}
\end{pmatrix} 
= \exp(\theta \ii H^\vee_\alpha) 
$$
It follows that
$$
\partial( \sigma^\vee_\alpha ) = [e^{\theta \ii}\mapsto 
\exp(\theta \ii H^\vee_\alpha) ] \in \pi_1(U_\T)
$$
which correponds to $\ii H^\vee_\alpha$ mod $\Gamma^\vee_\T$ under the natural isomorphism.  
\end{proof}

The real root decomposition of the isotropy Lie algebra is
$$
\u_\T = \t \oplus \sum_{\alpha \in \Pi^+_\T} \u_\alpha
$$
By the real root decomposition of $\u$ it follows that $\u_\T$ is the centralizer in $\u$ of the torus
\begin{equation}
\label{eq:toro-isotropia}
\t_\T = \{ H \in \t: \, \alpha(H) = 0 \quad \forall \alpha \in \T \}
\end{equation}
Since $U_\T$ is connected, it is the centralizer of $\t_\T$ in $U$.
Note that if $\alpha \in \Pi_\T$ then $U(\alpha) \subset U_\T$.

\begin{lemma}
\label{lema:isotropia-rank1}
If $\alpha \not\in \Pi_\T$ then $U(\alpha) \cap U_\T = T(\alpha)$
\end{lemma}
\begin{proof}
Let $T(\alpha)$ be the centralizer of $\ii H^\vee_\alpha$ in $U(\alpha)$, it is contained in $T$ and is a maximal torus of the rank 1 group $U(\alpha)$.
We have that 
$
U(\alpha) \cap T = T(\alpha)
$,
since an $u \in U(\alpha) \cap T$ centralizes $ \ii H^\vee_\alpha \in \t$, so that $u \in T(\alpha)$, and the other inclusion is immediate.

Now consider the orthogonal subspace to $H_\alpha$ given by $\t_\alpha = \{ H \in \t: \, \alpha(H) = 0 \}$.  If a root $\beta$ is such that $\beta|_{\t_\alpha} = 0$ then $H_\beta$ is a multiple of $H_\alpha$ so that $\beta = \pm \alpha$, since the root system is reduced.  It follows from the real root decomposition of $\u$ that the centralizer of $\t_\alpha$ in $\u$ is $\t \oplus \u_\alpha$.  Denote by $U_\alpha$ the centralizer of $\t_\alpha$ in $U$, its Lie algebra is then $\t \oplus \u_\alpha$.  It follows that $U_\alpha \cap U_\T$ is the centralizer of the torus $\t_\alpha + \t_\T$, hence it is connected.  By the real root decomposition of $\u$ and $\u_\T$, it follows that the centralizer of $\t_\alpha + \t_\T$ in $\u$ is
$$
(\t \oplus \u_\alpha) \cap (\t \oplus \sum_{\alpha \in \Pi_\T} \u_\alpha) = \t
$$
since $\alpha \not\in \Pi_\T$. It follows that
$
U_\alpha \cap U_\T = T
$.
Since $\u(\alpha) \subset \t \oplus \u_\alpha$, we have that $U(\alpha) \subset U_\alpha$, and since $T(\alpha) \subset T \subset U_\T$, it follows from the previous intersections that
$$
T(\alpha) \subset U(\alpha) \cap U_\T \subset U(\alpha) \cap ( U_\alpha \cap U_\T) = 
U(\alpha) \cap T = T(\alpha).
$$
as claimed. 
\end{proof}

The next result is the main result of this section.
\begin{theorem}
\label{teo:pi2}
\begin{enumerate}[(i)]
\item The dual roots $\alpha^\vee, \beta^\vee, \delta^\vee \in \Pi$ add up to 
$
\delta^\vee = \alpha^\vee + \beta^\vee$ {\rm mod} $R^\vee_\T
$ if, and only if, in $\pi_2(\F_\T)$ we have
$$
\sigma^\vee_\delta = \sigma^\vee_\alpha * \sigma^\vee_\beta
$$

\item If $\alpha \not\in \Pi_\T$ then $\sigma^\vee_\alpha$ is diffeomorphism onto a 2-sphere.

\item  With regards to the topology of $\F_\T$ and of the map $\sigma^\vee_\alpha$, we can assume that $U$ is compact semisimple and simply connected.

\item We have that $\{ \sigma^\vee_\alpha:  \, \alpha \in \Sigma - \T \}$ is a $\Z$-basis of $\pi_2(\F_\T)$.

\end{enumerate}
\end{theorem}
\begin{proof}
For item (i), by the isomorphism (\ref{isom:raizdual-coraiz}) it follows that 
$\delta^\vee = \alpha^\vee + \beta^\vee$ {\rm mod} $R^\vee_\T$ is equivalent to
$\ii H_\delta^\vee = \ii H_\alpha^\vee + \ii H_\beta^\vee \mod \Gamma^\vee_\T$.
Since the $\ii H_\bullet^\vee$ are coroot vectors and $\partial$ is injective, Proposition \ref{teo:bordo} then implies item (i).

For item (ii), the image of $\sigma^\vee_\alpha$ is the orbit $U(\alpha)o$ of the basepoint $o$. It has isotropy
$U(\alpha) \cap U_\T = T(\alpha)$, by the previous Lemma. Hence the image of $\sigma^\vee$ is diffeomorphic to $U(\alpha)/T(\alpha)$, which is a 2-sphere.
To show that $\sigma^\vee_\alpha$ is a diffeomorphism, from (\ref{eq:hopf-em-U}) we have the diagram
\begin{equation}
\begin{array}{rcl}
SU(2)/U(1) & \to & S^2 \quad \\
\downarrow \Psi \quad &  &  \downarrow  \sigma^\vee_\alpha \\
U(\alpha)/T(\alpha) & \to & U(\alpha)o 
\end{array}
\end{equation}
whose rows are diffeomorphisms, so that we only have to prove that the first column is a diffeomorphism. Since it is induced by the epimorphism $SU(2) \stackrel{\Psi}{\to} U(\alpha)$, we only have to prove that $\Psi^{-1}(T(\alpha)) = U(1)$.  We have already stablished the inclusion $\supseteq$, right before (\ref{eq:hopf-em-U}). Let then $u \in SU(2)$ be such that $\Psi(u) \in T(\alpha)$. Then $\Psi(u)$ centralizes $i H_\alpha^\vee = 
\psi\begin{psmallmatrix}
\ii &    \\
   & -\ii \\
\end{psmallmatrix}$
so that
$$
\psi
\left(
\Ad(u) \begin{psmallmatrix}
\ii &    \\
   & -\ii \\
\end{psmallmatrix}
\right)
=
\Ad(\Psi(u))
\psi\begin{psmallmatrix}
\ii &    \\
   & -\ii \\
\end{psmallmatrix}
=
\psi\begin{psmallmatrix}
\ii &    \\
   & -\ii \\
\end{psmallmatrix}
$$
and, since $\psi$ is an isomorphism of Lie algebras, if follows that $u$ centralizes $\begin{psmallmatrix}
\ii &    \\
   & -\ii \\
\end{psmallmatrix}$ so that it lies in $U(1)$, as claimed.

For item (iii), since the center $Z$ of $U$ is contained in $T$, it is contained in $U_\T$.
Consider the quotient groups $\ov{U} = U/Z$ and $\ov{U_\T} = U_\T/Z$.
Then $\ov{U}$ is connected semisimple.  Furthermore, its maximal torus $T/Z$ is contained in $\ov{U_\T}$ which is then a maximal rank subgroup.  Let $\ov{\Psi}: SU(2) \to \ov{U}$ be the homomorphism constructed as before, now for $\ov{U}$, and let $\ov{\pi}: U \to \ov{U}$ be the quotient homomorphism. Denote by $\ov{\exp}$ the exponential map of $\ov{U}$.
Since $\ov{\pi}( \exp ) =  \ov{\exp}( d\ov{\pi}_1  )$, it follows that the leftmost diagram below commutes
\begin{equation*}
\begin{tikzcd}
SU(2) \arrow["\Psi"]{r} \arrow["\ov{\Psi}"]{rd} & U \arrow["\ov{\pi}"]{d}\\ 
& \ov{U}
\end{tikzcd}
\qquad\qquad\qquad
\begin{tikzcd}
S^2 \arrow["\sigma^\vee_\alpha"]{r} \arrow["\ov{\sigma^\vee_\alpha}"]{rd} & \F_\T \arrow["\ov{\phi}"]{d}\\ 
& \ov{U}/\ov{U_\T} 
\end{tikzcd}
\end{equation*}
The quotient map $\ov{\pi}: U \to \ov{U}$ induces the diffeomorphism $\ov{\phi}: \F_\T \to \ov{U}/\ov{U_\T}$.  Let $\ov{\sigma^\vee_\alpha}: S^2 \to \ov{U}/\ov{U_\T}$ be the immersion constructed as before, now for $ \ov{U}/\ov{U_\T}$.  The commutativity of the leftmost diagram above implies that the righmost diagram also commutes.  

Thus, exchanging the pair $U, U_\T$ by $\ov{U}, \ov{U}_\T$, we can assume that $U$ is compact semisimple since the diffeomorphism $\ov{\phi}$ takes $\sigma^\vee_\alpha$ to $\ov{\sigma^\vee_\alpha}$.
Now $U$ being compact semisimple, it has a compact semisimple covering $\wt{U}$ with covering homomorphism $\wt{\pi}: \wt{U} \to U$.  Consider the subgroup $\wt{U}_\T = \wt{\pi}^{-1}( U_\T )$. It is closed, since $U_\T$ is closed and $\wt{\pi}$ is continuous, and it is of maximal rank in $\wt{U}$, since this is a local property.  Let $\wt{\Psi}: SU(2) \to \wt{U}$ be the homomorphism constructed as before, now for $\wt{U}$. 
Denote by $\wt{\exp}$ the exponential map of $\wt{U}$.  Since $\wt{\pi}( \wt{\exp} ) =  \exp$, it follows that the leftmost diagram below commutes
\begin{equation*}
\begin{tikzcd}
& \wt{U} \arrow["\wt{\pi}"]{d}
\\ 
SU(2) \arrow["\Psi"]{r} \arrow["\wt{\Psi}"]{ru} & U
\end{tikzcd}
\qquad\qquad\qquad
\begin{tikzcd}
& \wt{U}/\wt{U}_\T \arrow["\wt{\phi}"]{d}
\\ 
S^2 \arrow["\sigma^\vee_\alpha"]{r} \arrow["\wt{\sigma^\vee_\alpha}"]{ru} & U
\end{tikzcd}
\end{equation*}
The covering map $\wt{\pi}: \wt{U} \to U$ induces the diffeomorphism $\wt{\phi}: \wt{U}/\wt{U}_\T \to \F_\T$.  Let $\wt{\sigma^\vee_\alpha}: S^2 \to \wt{U}/\wt{U}_\T$ be the immersion constructed as before, now for $ \wt{U}/\wt{U}_\T$.  The commutativity of the leftmost diagram above implies that the righmost diagram also commutes. 
Thus, exchanging the pair $U, U_\T$ by $\wt{U}, \wt{U}_\T$, we can assume that $U$ is compact semisimple and simply connected since the diffeomorphism $\wt{\phi}$ takes $\wt{\sigma^\vee_\alpha}$ to $\sigma^\vee_\alpha$. 

For item (iv), we use the previous item to assume that $U$ is simply connected so that, by the exact sequence (\ref{eq:seqexata}), the boundary map is an isomorphism. By the no torsion part of Proposition \ref{prop:imagem}, it follows that $\alpha^\vee$, $\alpha \in \Sigma - \T$, projects to a $\Z$-basis of $R^\vee/R^\vee_\T$.  By the isomorphism (\ref{isom:raizdual-coraiz}) it follows that $\ii H^\vee_{\alpha}$, $\alpha \in \Sigma - \T$, projects to a $\Z$-basis of $\Gamma^\vee/\Gamma^\vee_\T$.  Proposition \ref{teo:bordo} then implies item (iv).
\end{proof}

It follows that $\pi_2(\F_\T) \simeq \Gamma^\vee/\Gamma^\vee_\T$ where we do not have a canonical isomorphism (since the boundary depends on the choice of $U$) but we do have a preferred $\Z$-basis.  By the isomorphism (\ref{isom:raizdual-coraiz}) it follows that 
\begin{equation}
\label{isom:raizdual-raizdual}
\pi_2(\F_\T) \simeq R^\vee/R^\vee_\T
\end{equation}
where, for roots $\alpha$, $\beta$, we have that $\sigma_\alpha^\vee = \sigma_\beta^\vee$ in $\pi_2(\F_\T)$ iff $\alpha^\vee = \beta^\vee \mod R^\vee_\T$.

\begin{remark}
Note that in items (i)-(iii) of the previous Theorem $\alpha$ is not necessarily a simple root.
For a simple root $\alpha$, the image of $\sigma_\alpha^\vee$ is a, so called, minimal Schubert variety of $\F_\T$.  
\end{remark}

\begin{remark}
Following these lines, the computations in Remark \ref{remark:quat} may be of use to describe 
the $\pi_4$ of {\em quaternionic} flag manifolds.  
\end{remark}


\section{Invariant geometry}

In this section we relate the well known invariant riemannian geometry of a flag manifold $\F_\T = U/U_\T$, which has a nice description in terms of the roots $\Pi$ of $U$ and the roots $\Pi_\T$ of the isotropy $U_\T$, with the spheres embeedings $\sigma^\vee_\alpha$ on $\F_\T$ of the previous section.

An $U$-invariant riemmanian metric of $\F_\T$ furnishes an inner product on the tangent space $T_o(\F_\T)$ at the basepoint  $o$ that is invariant by the isotropy representation of $U_\T$.  Reciprocally, once we fix a $U_\T$ -invariant inner product on $T_o(\F_\T)$, we can use the action of $U$ to spread it around $\F_\T$ and obtain a $U$-invariant riemmanian metric. It follows that a $U$-invariant riemmanian metric of $\F_\T$ is uniquely determined by the $U_\T$-invariant inner product on $T_o(\F_\T)$.
Note that the projection $U \to \F_\T$ differentiates at the identity to the linear epimorphism
\begin{equation}
\label{eq:esptg}
\u \to T_o(\F_\T)
\end{equation}
whose kernel is the isotropy subalgebra $\u_\T$.  Fix in $\u$ an $\Ad(U)$-invariant inner product $\prod{\cdot,\cdot}$. 
Consider $\m_\T = \u_\T^\perp$ the orthogonal complement in $\u$ of $\u_\T$.  The restriction of the above epimorphism to $\m_\T$ then canonically identifies it with $T_o(\F_\T)$.
The restriction of $\prod{\cdot,\cdot}$ to $\m_\T$ then furnishes an inner product in the tangent space.  Since $\m_\T$ is $U_\T$-invariant, the isotropy representation becomes $U_\T \to {\rm O}(\m_\T)$ and is given by $\Ad|_{\m_\T}$, the restriction to $\m_\T$ of the adjoint representation of $U_\T$.  

Now let $\alpha \not\in \Pi_\T$ and denote by $S^2(\alpha)$ the 2-sphere given by the image of the embeeding $\sigma^\vee_\alpha$. It is the orbit of the basepoint $o$ by $U(\alpha)$, whose tangent space at $o$ is the real root space $\u_\alpha$ and whose isotropy is, by Lemma \ref{lema:isotropia-rank1}, $U(\alpha) \cap U_\T = T(\alpha)$, which acts by rotations in $\u_\alpha$.  
Now, the restriction to $\u_\alpha$ of an $U_\T$-invariant inner product of $\m_\T$ furnishes another $T(\alpha)$-invariant inner product on $\u_\alpha$ which is, thus, a positive multiple by $\lambda_\alpha$ of the previous one, 
by the uniqueness of an inner product invariant by rotations.
The diameter of $S^2(\alpha)$ in the corresponding invariant metric is the previous diameter divided by $\lambda_\alpha$.
From the real root space decomposition of $\u$ and $\u_\T$, since distinct real root spaces are orthogonal under any $U$-invariant inner product, it follows that 
\begin{equation}
\label{eq:decompos-raizes-tg}
\m_\T = \sum_{\alpha \not\in \Pi_\T} \u_{\alpha}
\end{equation}
so that the various parameters $\lambda_\alpha$, $\alpha \not\in \Pi_\T$, that determine the diameter of the 2-spheres $S^2(\alpha)$, also determine the invariant metric. Neverthless, these parameters cannot in general be chosen independently since the isotropy action of $U_\T$ may take some $\u_{\alpha}$ to another $\u_\beta$ so that $\lambda_\alpha = \lambda_\beta$, since it acts by isometries.

\begin{example}
The $SO(3)$-invariant metrics of the 2-sphere $S^2$ are multiples of the round metric of the embeeding $S^2 \subset \R^3$.
In fact, let $\R^3$ be spanned by $\ii, \j, \k$, with the cross product this is the Lie algebra of $SO(3)$. The isotropy subgroup of the basepoint $\k$ is ${\rm SO}(2)$ and the isotropy subalgebra is the $\k$ axis. Fix in $\R^3$ the usual inner product. It follows that $\m = \k^\perp =$ the $\ii,\j$ plane $\R^2$, on which ${\rm SO}(2)$ acts by rotation.  The only rotation-invariant inner product on this plane are the positive multiples of the inner product of $\R^2$, which is the restriction of the inner product of $\R^3$.  The latter is precisely the inner product induced by the the embeeding $S^2 \subset \R^3$ in the tangent plane to $\k$, which proves our claim.

The same conclusion applies to the ${\rm SU}(2)$-invariant riemannian metrics of the Riemman sphere, when identified with the euclidean sphere $S^2$ as in (\ref{eq:SU2-SO3}).  This is because the adjoint action of ${\rm SU}(2)$ on its Lie algebra is the canonical action of the ${\rm SO}(3)$ on $\R^3$.  

\end{example}
%

Being an orthogonal representation, $\m_\T$ decomposes as the sum of mutually orthogonal and irreducible representations
\begin{equation}
\label{eq:decompos-esptg}
\m_\T = \m_1 \oplus \m_2 \oplus \cdots \oplus \m_k
\end{equation}
We call this an isotropy decomposition and each $\m_i$ an {\em isotropy component}.  The fundamental result is then Siebenthal's theorem (Théorème 7.1 of \cite{siebenthal}), which proves that the isotropy components are unique (apart from reordering) irreducible {\em inequivalent} representations and describes each one of them in terms of roots.   
Now, an $\Ad(U_\T)$-invariant inner product in $\m_\T$ is given by
$\prod{\Lambda X, Y}$, where $\Lambda: \m_\T \to \m_\T$ is a symmetric, positive linear operator that commutes with the adjoint action of $U_\T$.
From decomposition (\ref{eq:decompos-esptg}) and Schur's Lemma it follows that $\Lambda|_{\m_i} = \lambda_i$ times the identity of $\m_i$.  From the inequivalence of the isotropy components if follows that the $\lambda_i$'s are not related.  Thus, the $\Ad(U_\T)$-invariant inner products in $\m_\T$, and the corresponding $U$ invariant riemannian metrics of $\F_\T$, are described by the $k$ positive parameters $\lambda_1, \ldots, \lambda_k$, one for each isotropy component.

We now recall Siebenthal's description of the isotropy components $\m_i$.  
Denote by $R_\T$ the root group of the isotropy, given by the $\Z$-span of the roots $\Pi_\T$.
We consider the quotient group $\h^*/R_\T$ so that two functionals $\alpha,\, \beta \in \h^*$ are equivalent mod $R_\T$ when their difference is in $R_\T$. 
Partitioning the union of the positive roots $\Pi^+$ with the zero functional according to the equivalence classes mod $R_\T$ we get
$$
\Pi^+ \cup \{ 0 \} = \Pi_0 \cup \Pi_1 \cup \cdots \cup \Pi_s
$$
where the residue class of $0$ is
$$
\Pi_0 = \Pi^+_\T \cup \{0\}
$$ 
The residue classes mod $R_\T$ characterize the isotropy components by  
$$
\m_i = \sum_{\alpha \in \Pi_i} \u_{\alpha}
$$
where $i = 1,2,\ldots, k$.

\begin{example}
In type A, the mod $R_\T$ equivalence classes of positive roots and the isotropy components are related to block decompositions of the Lie algebra.  Consider for example the partial flag manifold $\F_\T = {\rm SU}(9)/{\rm S}({\rm U}(3)\times {\rm U}(1) \times {\rm U}(2))$. Then its simple roots $\T$ corresponds to the three intervals of simple roots given by the black dots above the diagonal in the picture below.
\begin{center}
       \def\svgwidth{5cm}
       \input{blocos-isotropia_pdf.pdf}
\end{center}
The other hollow black dots in the blocks right above $\T$ correspond to the other positive roots in the isotropy. The colored dots outside the blocks on the diagonal correspond to positive roots outside the isotropy. Their block decomposition correspond to the equivalence classes mod $\T$, say the red ones give $\Pi_1$, the green ones give $\Pi_2$ and the blue ones give $\Pi_3$.  The corresponding isotropy components $\m_1$, $\m_2$, $\m_3$ are then given by picturing the above picture as a block decomposition of anti-hermitian  matrices, where the upper triangular block is appropriately reflected onto the lower triangular block.
\end{example}

\begin{remark}
In the maximal flag manifold we have $\T = \emptyset$, so that each $R_\T = \{ 0 \}$, each residue class $\Pi_i$ is a singleton, each $\m_i$ is a real root space, and then (\ref{eq:decompos-raizes-tg}) is already the isotropy decomposition.  

In partial flag manifolds where $\T \neq \emptyset$, apart from $\Pi_0 - \{0\}$, the other residue classes do not, in general,  define root subsystems. For example, if there exists two $\Theta$-equivalent roots $\alpha, \beta$ in $\Pi_i$, $i \neq 0$, with Cartan integer $\prod{\alpha, \beta^\vee}=1$ then $r_\beta(\alpha) = \alpha - \beta$ is a root that lies in $R_\T$, and so in $\Pi_\Theta$, thus outside $\Pi_i$.
\end{remark}

\subsection{Equiharmonic spheres}
\label{equiharmonic-spheres}

We now establish an interesting relation between the spheres $S^2(\alpha)$ and harmonic maps into $\F_\T$. Consider $M^2$ a compact Riemann surface equipped with a metric $g$, let $(N,h)$ be a compact Riemannian manifold and $\phi:(M^2,g)\to (N,h)$ a differentiable map. The {\em energy} of $\phi$ is given by 
$$
E(\phi)=\frac{1}{2}\int_{M}{|d\phi|^2\, \omega_g},
$$ 
where $\omega_g$ is the volume measure defined by the metric $g$ and $|d\phi|$ is the Hilbert-Schmidt norm of $d\phi$. The differentiable map $\phi$ is {\em harmonic} if it is a critical point of the energy functional. Examples of harmonic maps are the minimal immersions and the totally geodesic immersions.

\begin{definition}
A map $\phi:M^2\to \F$ is called {\em equihamonic map} if it is harmonic with respect {\em any} invariant metric on $\F$. 
\end{definition}

From now on we restrict ourselves to the case $M^2=S^2$, the 2-sphere.  We will use also the concept of {\em homogeneous equigeodesics}: curves of the form $\gamma(t)=(\exp tX)\,  T$, $X\in\mathfrak{m}_\T$ that are geodesics with respect to any invariant metric. The vector $X\in \mathfrak{m}_\T$ is called {\em equigeodesic vector}. One can think equigeodesics as equihamonic maps whose the domain is 1-dimensional. In \cite{CGN} it is proved that any flag manifolds admits such curves and each isotropy component of ($\ref{eq:decompos-esptg}$) consists of equigeodesic vectors, that is, for all $X\in \mathfrak{m}_i$ the curve $\gamma(t)=(\exp tX)\,  T$ is an equigeodesic on $\F_\T$. 

\begin{proposition}
For each root $\alpha \not \in \Pi_\T$ the sphere $S^2(\alpha)$ is equihamonic, that is, the map $\sigma^\vee_\alpha: S^2 \to \F_\T$ is equiharmonic. In particular $\pi_2(\F_\T)$ is generated by equiharmonic 2-spheres. 
\end{proposition}
\begin{proof}
We will show that $S^2(\alpha)$ is totally geodesic, independent of the $U$ invariant metric on $\F_\T$. Fix an $U$-invariant metric on $\F_\T$ (and therefore an $\Ad(T)$-invariant scalar product on $\mathfrak{m}_\T$). Keep in mind the decomposition (\ref{eq:decompos-esptg}) associated with the isotropy representation. It is clear that the tangent space of  $S^2(\alpha)$ at $o$ is $\mathfrak{u}_\alpha$.
Since $S^2(\alpha)$ is an $U(\alpha)$-orbit with isotropy $U(\alpha) \cap U_\T  = T(\alpha)$ (see Lemma \ref{lema:isotropia-rank1}) it follows that the $U$-invariant metric induces $T(\alpha)$-invariant inner product on $\u_\alpha$.  Since $T(\alpha)$ acts in $\u_\alpha$ by rotation, it follows that this inner product on $\u_\alpha$ is a multiple of the Cartan-Killing inner-product of $\u(\alpha)$.  It follows that the induced metric on $S^2(\alpha)$ is a multiple of the round metric and every geodesic on $S^2(\alpha)$ is of the form $\gamma(t)=(\exp tX) T(\alpha)$. We remark that $\gamma$ is also a geodesic on $\F_\T$ since the vector $X\in\mathfrak{u}_\alpha$ is an equigeodesic vector for $\F_\T$.
Therefore, a geodesic on $S^2(\alpha)$ is also a geodesic on $\F_\T$ so that $S^2(\alpha)$ is totally geodesic with respect to the fixed invariant metric on $\F_\T$. 
Since the fixed invariant metric is arbitrary, the result follows.
\end{proof}

\subsection{$\T$-rigid roots}
\label{theta-rigid}

To interpret the invariant metrics of $\F_\T$ geometrically, denote by $S^2(\alpha)$ the embedded 2-sphere in $\F_\T$ given by the image of $\sigma^\vee_\alpha$, for a root $\alpha \not \in \Pi_\T$.

\begin{proposition}
The invariant metrics of $\F_\T$ are given by choosing the diameters of the 2-spheres $S^2(\alpha)$, for the roots $\alpha \not\in \Pi_\T$, where two spheres $S^2(\alpha)$, $S^2(\beta)$ must be given the same diameter whenever their corresponding roots $\alpha, \beta$ are congruent {\rm mod} $R_\T$.  In this case we say that these two spheres have the same invariant geometry.
\end{proposition}

This partitions the spheres $S^2(\alpha)$ in classes with the same invariant geometry under any invariant metric of the ambient space $\F_\T$.
Clearly, a sufficient condition for $S^2(\alpha)$ and $S^2(\beta)$ to have the same invariant geometry is that there exists $u \in U$ such that $u S^2(\alpha) = S^2(\beta)$.  The next result shows that this condition may be checked purely on the root system.  
Denote by $W$ be the Weyl group of $U$ and by $W_\T$ be the Weyl group of the isotropy $U_\T$.

\begin{proposition}
\label{propos:isometria-esferas}
There exists $u \in U$ such that $u S^2(\alpha) = S^2(\beta)$ in $\F_\T$ iff there exists $w \in W_\T$ such that $w^*\alpha = \beta$.  In particular, $\alpha$ and $\beta$ have the same length.
\end{proposition}
\begin{proof}
Let $w^*\alpha = \beta$, $w \in W_\T$. There exists $u \in U_\T$ such that $\Ad(u)|_{\t} = w$.  Then $\Ad(u)\u_{\alpha} = \u_{\Ad(u)^*\alpha} = \u_{w^*\alpha} = \u_{\beta}$ and also $\Ad(u) \ii H_\alpha = \ii H_{\Ad(u)^*\alpha} = \ii H_{w^*\alpha} =  \ii H_{\beta}$. It follows that $\Ad(u) \u(\alpha) = \u(\beta)$.  Exponentiating we get $ u U(\alpha) u^{-1} \, o = u U(\alpha) \, o = u U(\beta) \, o$, as claimed, since $u \, o = o$.

Now let $S^2(\beta) = U(\beta) \, o$.  There exists then $v \in U(\beta)$ such that $u \, o = v \, o$.  Thus $ v^{-1} u \in U_\T$ is such that
$v^{-1} u  S^2(\alpha) =  v^{-1} S^2(\beta) = S^2(\beta)$.  Hence, replacing $u$ by $v^{-1} u$, we can assume that $u \in U_\T$ which, together with
$u S^2(\alpha) = S^2(\beta)$, imply that $\Ad(u) \u_{\alpha} = \u_\beta$.  It follows that $\Ad(u) \g_{\alpha} = \g_\beta$ and, since $\Ad(u) \g_{\alpha} = \g_{\Ad(u)^*\alpha}$, 
we have that $\Ad(u)^*\alpha = \beta$. It follows that $\Ad(u)H_\alpha = H_\beta$, so that $\Ad(u) \ii H_\alpha = \ii H_\beta$.
Since both $\ii H_\alpha, \ii H_\beta \in \t$, there exists $w \in W$ such that $\Ad(u) \ii H_\alpha = w \ii H_\alpha = \ii H_{w^*\alpha} = \ii H_\beta$, so that $w^*\alpha = \beta$. 
\end{proof}

It is interesting to note that there may spheres $S^2(\alpha)$, $S^2(\beta)$ with the same invariant geometry but with roots $\alpha, \beta$ of different lengths (see Example \ref{exemplog2})

The 2-spheres $\sigma^\vee_\alpha$ that induce generators of $\pi_2(\F_\T)$, a homotopical object, also generate the invariant geometry of $\F_\T$, a geometrical object.   When do these homotopical generators of $\pi_2(\F_\T)$ carry the same invariant geometry?  This reduces to a question about root systems, as follows.
For a subset of roots $A \subset \Pi$, denote by $A^\vee = \{ \alpha^\vee:\, \alpha \in A\}$.  

\begin{remark}
Note that in general $R^\vee_\T$ is not the same as $(R_\T)^\vee$, since $R_\T$ is the $\Z$-span of $\T$, $R^\vee_\T$  is the $\Z$-span of $\T^\vee$ and the map $\alpha \mapsto \alpha^\vee$ is in general not linear.
Thus, $\alpha = \beta \mod R_\T$ does not imply in general that  $\alpha^\vee = \beta^\vee \mod R^\vee_\T$ and vice-versa, see the examples in the next Section.
\end{remark}

Denote by $\Pi_\T(\alpha)$ the mod $R_\T$ class of $\alpha$ in $\Pi$ and by $\Pi^\vee_\T(\alpha^\vee)$ the mod $R^\vee_\T$ class of $\alpha^\vee$ in $\Pi^\vee$.  
We say that a root $\alpha$ is {\em $\T$-rigid} when $\alpha \not\in \Pi_\T$ and
$$
\Pi_\T(\alpha)^\vee \subseteq \Pi^\vee_\T(\alpha^\vee)
$$
equivalently, when the 2-spheres $S^2(\beta)$ that have the same invariant geometry as $S^2(\alpha)$ are in the same homotopy class as $S^2(\alpha)$.  
Note that $\T$-rigidity is a property of the whole residue class $\Pi_\T(\alpha)$.
By the duality $(\alpha^\vee)^\vee = \alpha$ it follows that $\alpha^\vee$ is $\T^\vee$-rigid when
$$
\Pi^\vee_\T(\alpha^\vee)^\vee \subseteq   \Pi_\T(\alpha)
$$
equivalently, when the 2-spheres $S^2(\beta)$ that have the same homotopy class as $S^2(\alpha)$ have the same invariant geometry as $S^2(\alpha)$.

In Section \ref{actionW} we characterize $\T$-rigid roots by using the action the Weyl group $W_\T$ on the residue classes of roots mod $R_\T$.  Below we state the main consequences of these results to the invariant geometry of flag manifolds.

\begin{theorem} For a flag manifold $\F_\T$ with root system $\Pi$ and isotropy root system $\Pi_\T$, we have the following.
\begin{enumerate}[(i)]
\item If root $\beta \in \Pi_\T(\alpha)$ has the same length as $\alpha$, then the 2-sphere $S^2(\beta)$ has the same invariant geometry and same homotopy class as $S^2(\alpha)$.

\item A root $\alpha$ is $\T$-rigid 
iff the roots of the residue class $\Pi_\T(\alpha)$ have the same length,  
iff  each sphere with the same invariant geometry as $S^2(\alpha)$ is the translation of $S^2(\alpha)$ by an isometry of $U$.

\item The homotopy classes of all the 2-spheres $S^2(\alpha)$, $\alpha \not\in \Pi_\T$, coincide with their invariant geometry classes iff either the root system $\Pi$ is simply laced and $\Theta$ is arbitrary or else the root system is not simply laced and $\Theta = \emptyset$.
\end{enumerate}
\end{theorem}
\begin{proof}
For item (i), Proposition \ref{propos1} gives that $W_\T$ acts transitively in the roots of same length in $\Pi_\T(\alpha)$, thus $\beta = w^*\alpha$ for some $w \in W_\T$.  Now Proposition \ref{resid-invar} gives the invariance of $\Pi_\T(\alpha)$ and $\Pi^\vee_\T(\alpha^\vee)$ by $W_\T = W_{\T^\vee}$.  Thus $\beta^\vee = (w \alpha)^\vee = w \alpha^\vee$, so that $\beta^\vee$ is in the same mod $R^\vee_\T$ class of $\alpha^\vee$.  Hence $S^2(\beta)$ has the same homotopy class as $S^2(\alpha)$, as claimed.

Item (ii) follows from Theorem \ref{propos3} items (ii) and (i), using Proposition \ref{propos:isometria-esferas} for the last conclusion.
Item (iii) is an immediate consequence of Corollary \ref{corol3}.
\end{proof}

For the flag manifolds of {\em simple} compact Lie groups, it follows that the invariant geometry classes of spheres coincide with their homotopy classes iff the root system is simply laced $A$, $D$, $E$ with arbitrary $\Theta$, thus any flag manifold of these types,  or of type $B$, $C$, $F$, $G$ with empty $\Theta = \emptyset$, thus only the maximal flag manifold of these types.

\section{Action of $W$ on the isotropy components}
\label{actionW}

The results of this section are stated and proved using purely the combinatorics of  root systems. We consider, more generally, nonreduced root system.
Recall the partition 
$$
\Pi \cup \{0\} = \Pi_0 \cup \Pi_1 \cup \cdots \cup \Pi_s
$$
into mod $R_\T$ residue classes considered in the last section, which depends solely on the root data $\Pi \supseteq \Pi_\T$, where $\Pi_0$ is the residue class of zero.  
Let $W_\T$ be the Weyl group of the root system $\Pi_\T$.  In this section we characterize when is $W_\T$ is transitive on each nonzero residue class.  It already leaves invariant each residue class mod $R_\T$.

\begin{proposition}
\label{resid-invar}
Each residue class of roots mod $R_\T$ is $W_\T$-invariant.
\end{proposition}
\begin{proof}
Recall the annihilator $\t_\Theta$ of $\Theta$ in $\t$ given by (\ref{eq:toro-isotropia}). Then $\phi \in \t^*$ is such that the restriction $\phi|_{\t_\Theta} = 0$ iff $\phi$ is a linear combination of roots in $\Theta$.  It follows that two roots satisfy $\alpha = \beta$ mod $R_\T$ iff their restrictions $\alpha|_{\t_\Theta} = \beta|_{\t_\Theta}$ coincide.   Since $W_\Theta$ centralizes $\t_\Theta$, it follows that each residue class mod $R_\T$ is invariant by $W_\Theta$. 
\end{proof}

Recall the duality bijection $\Pi \to \Pi^\vee$, $\alpha \mapsto \alpha^\vee = 2\alpha/\prod{\alpha,\alpha}$, which is in general not linear, even tough $(\alpha^\vee)^\vee = \alpha$.  It is only linear for the simply laced root systems.  We fix a subset of simple roots $\T \subset \Sigma$.  Recall that $W$ is generated by the simple reflections $r_\alpha$, $\alpha \in \Sigma$, while $W_\T$ is generated by the simple reflections $r_\alpha$, $\alpha \in \T$.  We have that the Weyl group of $\Pi_\Theta$ and its dual root system $\Pi^\vee_{\Theta}$ coincide $W_\Theta = W_{\Theta^\vee}$, since for the corresponding reflections we have $r_{\alpha} = r_{\alpha^\vee}$.   Since the Weyl group elements are isometries, we have that duality respects the Weyl group actions, that is, $(w\alpha)^\vee = w \alpha^\vee$ for all $w \in W$.  

Recall that the possible squared lengths of roots in a root system are 1, 2 and 4, called short, long and longer roots, respectively.  Simple roots are either short or long.  In a doubly laced root system, for a short root $\alpha$ we have that $\alpha^\vee = 2 \alpha$, for a long root $\alpha$ we have that $\alpha^\vee = \alpha$, and for a longer root we have that $\alpha^\vee = \alpha/2$.  For a triply laced root system, exchange 2 for 3 and exclude the longer roots. It follows that duality $\Sigma \to \Sigma^\vee$ exchanges short and long simple roots.

\begin{example}
The following picture illustrates the nonzero classes mod $R_\T$ for the non-reduced root system of type $BC_2$, only the positive roots are displayed.

\begin{center}
       \def\svgwidth{12cm}
\begingroup%
  \makeatletter%
  \providecommand\color[2][]{%
    \errmessage{(Inkscape) Color is used for the text in Inkscape, but the package 'color.sty' is not loaded}%
    \renewcommand\color[2][]{}%
  }%
  \providecommand\transparent[1]{%
    \errmessage{(Inkscape) Transparency is used (non-zero) for the text in Inkscape, but the package 'transparent.sty' is not loaded}%
    \renewcommand\transparent[1]{}%
  }%
  \providecommand\rotatebox[2]{#2}%
  \ifx\svgwidth\undefined%
    \setlength{\unitlength}{1102.7319755bp}%
    \ifx\svgscale\undefined%
      \relax%
    \else%
      \setlength{\unitlength}{\unitlength * \real{\svgscale}}%
    \fi%
  \else%
    \setlength{\unitlength}{\svgwidth}%
  \fi%
  \global\let\svgwidth\undefined%
  \global\let\svgscale\undefined%
  \makeatother%
  \begin{picture}(1,0.39551553)%
    \put(0.05356171,0.36711822){\color[rgb]{0,0,0}\makebox(0,0)[lb]{\smash{$BC_2\,\,{\rm mod}\,\,R_\Theta$}}}%
    \put(0,0){\includegraphics[width=\unitlength,page=1]{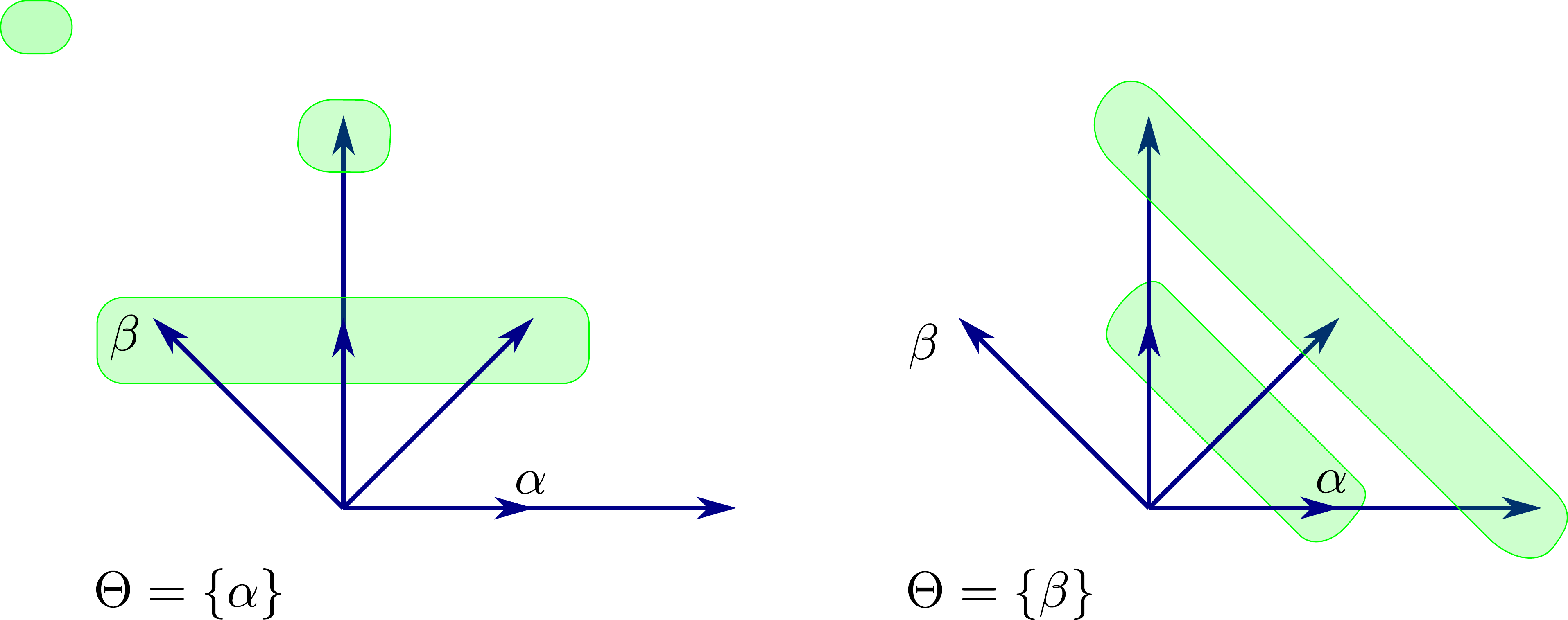}}%
  \end{picture}%
\endgroup%

\end{center}

\end{example}

A sequence of roots $\beta_1, \ldots, \beta_k$ connecting $\beta_1$ to $\beta_k$ is a $\Theta$-sequence when $\beta_{i+1} - \beta_i \in\Pi_\Theta$ $(i=1,2,\ldots,k-1)$. A subset of roots is $\Theta$-connected when all pair of roots in the subset can be connected by a $\Theta$-sequence. Siebenthal's fundamental result is that each nonzero residue class $\Pi_i$ is $\Theta$-connected (Proposition 2.1 of \cite{siebenthal}).  

What about the Weyl group of the root system, can it be used to connect the roots in each residue class? Surely no, if the residue class has roots of two distinct lengths.  Also no, for the residue class of zero, since the diagram of $\Theta$ can be disconnected.  Apart from that, we show in the next result that the answer is positive, by adapting Siebenthal's proof.


\begin{proposition}
\label{propos1}
$W_\Theta$ is transitive in the roots of same length in each nonzero residue class mod $R_\T$.
\end{proposition}
\begin{proof}
Fix a possible length $L$ of the roots, $L^2 =$ 1, 2 or 4.  Let $\beta_1 \neq \beta_2$ be distinct roots of equal length $L$ which are $\Theta$-equivalent.  Since $\beta_1$ and $\beta_2$ are distinct and have that same length, the only possibility that they are proportional is that $\beta_1 = - \beta_2$. By Cauchy-Schwartz we have that
$$
\prod{\beta_1 - \beta_2, \beta_1} = |\beta_1|^2 - \prod{\beta_2, \beta_1} 
> |\beta_1|^2 - |\beta_2||\beta_1| = 0 
$$
Writing $\Theta = \{ \alpha_1, \ldots, \alpha_n \}$, we have that 
$\beta_1 - \beta_2 = \sum_i n_i \alpha_i$, for $n_i \in \Z$.  It follows that
$$
\prod{\beta_1 - \beta_2, \beta_1} =  \sum_i n_i \prod{\alpha_i,\beta_1} > 0
$$
We claim that there exists $n_j \neq 0$ such that $\prod{\beta_1,\alpha_j} \neq 0$ and has the same sign of $n_j$, that is, $n_j \prod{\beta_1,\alpha_j} > 0$.  Indeed, suppose not, then 
$n_j  \prod{\beta_1,\alpha_j} \leq 0$ for all $j$ and adding up we get $\prod{\beta_1 - \beta_2, \beta_1} =  \sum_i n_i \prod{\alpha_i,\beta_1} \leq 0$, a contradiction. 
We thus get the root
\begin{eqnarray*}
r_{\alpha_j}(\beta_1) & = & \beta_1 - \prod{\beta_1,\alpha_j^\vee} \alpha_j \\
& = & \beta_2 + n_1 \alpha_1 + \cdots + (n_j - \prod{\beta_1,\alpha^\vee})\alpha_j + \cdots 
\end{eqnarray*}
where the coefficient of $\alpha_j$ in the above root has absolute value strictly smaller than that of $n_j$, since $n_j$ and $\prod{\beta_1,\alpha^\vee}$ have the same sign.  Also note that $r_{\alpha_j}(\beta_1)$ has the same length $L$ as $\beta_1$ and lies in the same mod $R_\T$ residue class as $\beta_1$.  Proceeding inductively, the coefficients in $\Theta$ decrease strictly and in a finite number of steps we get roots $\alpha_j, \ldots, \alpha_k$ in $\Theta$ such that  
$$
r_{\alpha_k} \cdots r_{\alpha_j} (\beta_1) = \beta_2 
$$
where $r_{\alpha_k} \cdots r_{\alpha_j}  \in W_\Theta$.
\end{proof}

It follows that a necessary and sufficient criterion for $W_\Theta$ to be transitive in a nonzero residue class mod $R_\T$ is that the roots of this class have the same length.  Next, we give some criteria to check this.  First, a long root criterion.
Note that, since $\Theta$ is a subset of simple roots, it can only contain short or long roots.

\begin{lemma}
\label{lemma2}
Let $\alpha$ be a short simple root. Then there exists a long root $\phi$ in $\Pi$ such that $\prod{\alpha, \phi} \neq 0$.
\end{lemma}
\begin{proof}
Since $\alpha$ is short simple, its connected component on the Dynkin diagram of $\Pi$ contains a long simple root. It follows that $\alpha$ must is connected to a long simple root $\beta$, more precisely, there exists a sequence of short simple roots $ \alpha = \alpha_0,\, \alpha_1,\, \ldots,\, \alpha_n$ such that $\prod{\alpha_i, \alpha_{i+1}^\vee} = -1$ and $k = \prod{\alpha_n,\beta^\vee} < 0$.  If $n=0$ then $\prod{\alpha,\beta^\vee} < 0$ 
and we are done.  
Else, denote by $r_i$ the reflection on the simple root $\alpha_i$ and consider
$$
\begin{array}{clll}
\phi_0 & = r_n(\beta) & = \beta - k \alpha_n \\
\phi_1 & = r_{n-1}(\phi_0) & = \beta - k( \alpha_n + \alpha_{n-1}) \\
\phi_2 & = r_{n-2}(\phi_1) & = \beta - k( \alpha_n + \alpha_{n-1} + \alpha_{n-2}) \\
\cdots \\
\phi_n & = r_{n}(\phi_{n-1}) & = \beta - k( \alpha_n + \ldots + \alpha_1 + \alpha ) 
\end{array}
$$
Then $\phi = \phi_n$ is a long root, since it is the image of $\beta$ by the Weyl group. Thus $\phi^\vee = \phi$ and then
$$
\prod{\alpha,\phi^\vee} = -k( \prod{\alpha, \alpha_1} + \prod{\alpha, \alpha} ) = -k( -1 + 2 ) = -k > 0
$$
since, by construction, $\alpha$ is orthogonal to $\beta,\, \alpha_n,\, \ldots, \alpha_2$.  This proves our claim.
\end{proof}

\begin{proposition}
\label{propos2}
The following are equivalent:
\begin{enumerate}[(i)]
\item $W_\Theta$ is transitive in each nonzero residue classes mod $R_\T$.
\item Each nonzero residue classes mod $R_\T$ has roots of the same length.
\item The roots in $\Theta$ are long.
\end{enumerate}
\end{proposition}
\begin{proof}
It is immediate that (i) implies (ii).  
To prove that (ii) implies (iii), suppose by contradiction that $\Theta$ contains a short root $\alpha$. Then by the previous Lemma there exists a long root $\phi$ such that 
the $\alpha$-string of roots through $\phi$ is non-empty.  Since $\phi$ is long and $\alpha$ is short, this $\alpha$-string contains both long and short roots, and since $\alpha \in \Theta$, this $\alpha$-string is contained in the the residue class of $\phi$ mod $R_\T$, which contradicts (ii).

To prove that (iii) implies (i), proceed as in 
Lemma 4.3, Item 1, of \cite{mauro-luiz}.
\end{proof}

\begin{remark}
\label{remark:simply-laced}
It follows that for simply laced root systems, which contain only long roots, $W_\Theta$ is always transitive in the nonzero residue classes mod $R_\T$, for arbitrary $\Theta$.
\end{remark}

Now we characterize transitivity in each residue class by using duality.  First, an example that will illustrate our next result.

\begin{example}
\label{exemplog2}
The following picture illustrates the nonzero classes mod $R_\T$ for the root system of type 
$G_2$, the nonzero classes mod $R^\vee_\T$  of its (isomorphic) dual $G_2^\vee$ and also the image of these residue classes under duality, for $\Theta = \{\alpha\}$.  Only the positive roots are displayed.

\begin{center}
       \def\svgwidth{13.5cm}
\begingroup%
  \makeatletter%
  \providecommand\color[2][]{%
    \errmessage{(Inkscape) Color is used for the text in Inkscape, but the package 'color.sty' is not loaded}%
    \renewcommand\color[2][]{}%
  }%
  \providecommand\transparent[1]{%
    \errmessage{(Inkscape) Transparency is used (non-zero) for the text in Inkscape, but the package 'transparent.sty' is not loaded}%
    \renewcommand\transparent[1]{}%
  }%
  \providecommand\rotatebox[2]{#2}%
  \ifx\svgwidth\undefined%
    \setlength{\unitlength}{1381.23788078bp}%
    \ifx\svgscale\undefined%
      \relax%
    \else%
      \setlength{\unitlength}{\unitlength * \real{\svgscale}}%
    \fi%
  \else%
    \setlength{\unitlength}{\svgwidth}%
  \fi%
  \global\let\svgwidth\undefined%
  \global\let\svgscale\undefined%
  \makeatother%
  \begin{picture}(1,0.35314122)%
    \put(0.0443698,0.33001398){\color[rgb]{0,0,0}\makebox(0,0)[lb]{\smash{$G_2\,\,{\rm mod}\,\,R_\Theta$}}}%
    \put(0,0){\includegraphics[width=\unitlength,page=1]{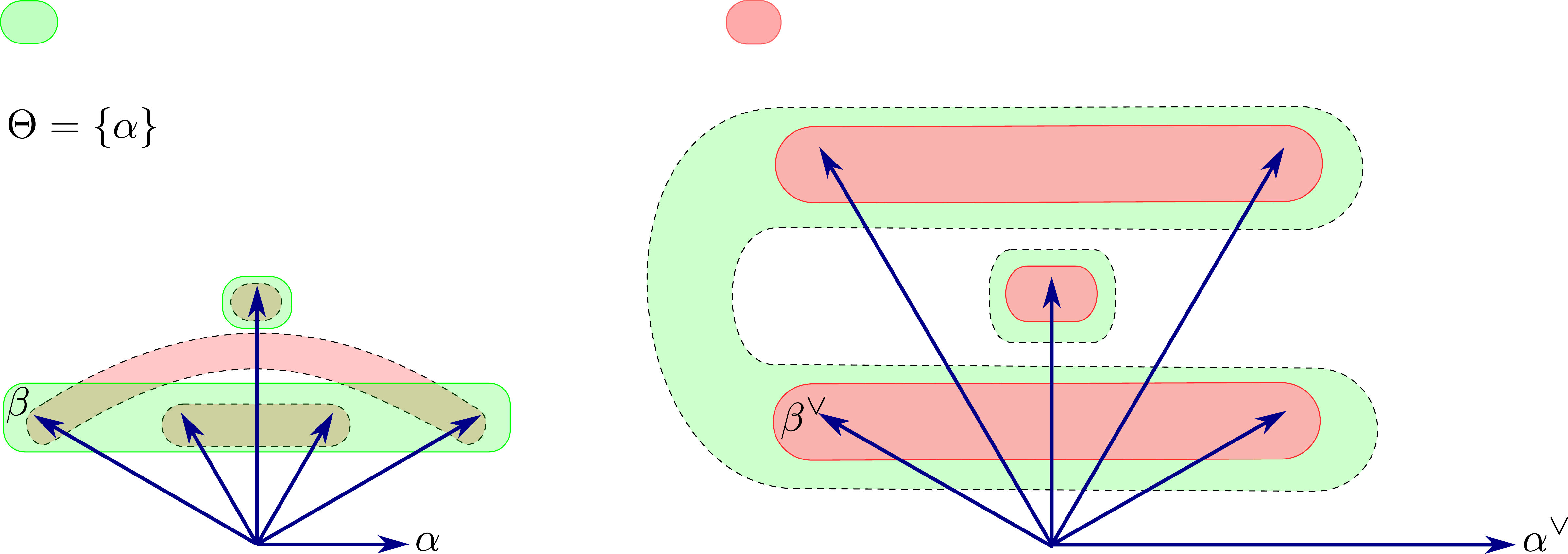}}%
    \put(0.5065639,0.33001398){\color[rgb]{0,0,0}\makebox(0,0)[lb]{\smash{$G_2^\vee\,\,{\rm mod}\,\,R^\vee_\Theta$}}}%
  \end{picture}%
\endgroup%

\end{center}

Note that $W_\Theta$ is not transitive in one of the two mod $R_\T$ classes and that their images by duality in the dual $G_2^\vee$ are not contained in the mod $R_\T^\vee$ classes. 
In the dual $G_2^\vee$, where $\alpha^\vee$ takes the role of the long root, note that $W_\Theta$ is now transitive in the three classes mod $R_\T^\vee$ and that their image by duality in the original $G_2$ are contained in the two mod $R_\T$ classes.  
\end{example}


\begin{theorem}
\label{propos3}
Given a root $\alpha \not\in \Pi_\T$, the following are equivalent:
\begin{enumerate}[(i)]
\item $W_\Theta$ is transitive in the residue class of $\alpha$ mod $R_\T$.
\item $\Pi_\T(\alpha)^\vee \subseteq \Pi^\vee_\T(\alpha^\vee)$.
\item The roots in $\Pi_\T(\alpha)$ have the same length.
\end{enumerate}
\end{theorem}

\begin{proof}
To prove that (i) implies (ii), first note that transitivity implies that the residue class of $\alpha$ is given by the orbit $W_\Theta \alpha = \Pi_\T(\alpha)$.  For the orbit we have
$$
(W_\Theta \alpha)^\vee = W_\Theta \alpha^\vee \subseteq \Pi^\vee_\T(\alpha^\vee)
$$
since duality respects the Weyl group action and since the mod $R_\T^\vee$ classes in $\Pi^\vee$ are $W_\Theta$-invariant.

To prove that (ii) implies (iii) first note that if a nonzero multiple $k \alpha$ belongs to the $\Pi_\T$, then $\alpha \in \Pi_\T$.  In fact, write $\alpha = \sum n_i \alpha_i$ for integers $n_i$ and simple roots $\alpha_i$, so that by linear independence of the simple roots $k\alpha = \sum_i kr_i \alpha_i \in \Pi_\T$ implies that $kr_i = 0$ for $\alpha_i \not\in \Theta$.  It follows that $r_i = 0$ for $\alpha_i \not\in \T$ and thus $\alpha \in \Pi_\T$.
Now, suppose that the residue class of $\alpha$ has roots of different lengths. Then, there exists a root
$\alpha + \gamma$  with length different from the length of $\alpha$, with $\gamma \in R_\T$.
We have the following possibilities.
\begin{enumerate}[(a)]
\item[] Suppose first that the root system is doubly laced.

\item $\alpha$ is long and $\alpha + \gamma$ is short:  it follows that
$(\alpha + \gamma)^\vee = 2 \alpha + 2 \gamma$ and $\alpha^\vee = \alpha$.
But from (ii) we get that 
$(\alpha + \gamma)^\vee = \alpha^\vee + \phi$, for some $\phi \in R^\vee_\T$, so that
$$
2 \alpha + 2 \gamma = \alpha + \phi
\quad
\text{and thus}
\quad
\alpha = \phi - 2\gamma
$$
Write $\gamma = \sum_i r_i \alpha_i$, with $r_i \in \Z$, $\alpha_i \in \T$.  Then $2\gamma = \sum_i r_i 2\gamma_i = \sum_i s_i \gamma^\vee_i$, where $s_i = r_i$ if $\gamma_i$ is short and $s_i = 2 r_i$ if $\gamma_i$ is long.  This shows that $2\gamma \in   R^\vee_\T$ and thus $\alpha \in R^\vee_\T$.  But then $\alpha^\vee = \alpha \in \Pi^\vee_\T \cap \Pi_\T$, so that $\alpha \in \Pi_\T$, a contradiction.

\item $\alpha$ is short and $\alpha + \gamma$ is long: it follows that
$(\alpha + \gamma)^\vee = \alpha + \gamma$ and $\alpha^\vee = 2\alpha$.
Again from (ii) we get that
$$
\alpha + \gamma = 2\alpha + \phi
\quad
\text{and thus}
\quad
\alpha = \phi - \gamma
$$
for some $\phi \in R^\vee_\T$.
Write $\phi = \sum_i r_i \alpha^\vee_i$, with $r_i \in \Z$, $\alpha_i \in \T$.
Since for the simple roots we have that $\alpha^\vee_i$ is either $\alpha_i$ or $2\alpha_i$, it follows that $\phi$ lies in $R_\T$. This shows that the root $\alpha$ lies in $R_\T$ and thus in $\Pi_\T$,  a contradiction.

\item[]
If the root system is triply laced we can replace 2 by 3 in the previous arguments and get that
$2 \alpha \in R_\T$.  By the above equation it follows that $\alpha  \in \Pi_\T$, a contradiction.

The next possibilities can only happen for doubly laced root systems:

\item $\alpha$ is long and $\alpha + \gamma$ is longer: it follows that
$(\alpha + \gamma)^\vee = (\alpha + \gamma)/2$ and $\alpha^\vee = \alpha$.
Again from (ii) we get that 
$$
(\alpha + \gamma)/2 = \alpha + \phi
\quad
\text{and thus}
\quad
\alpha = 2\phi - \gamma
$$
for some $\phi \in R^\vee_\T$.
As in item (b), we have that $\phi$ lies in $R_\T$.
This shows that the root $\alpha$ lies in $R_\T$ and thus in $\Pi_\T$,  a contradiction.

\item $\alpha$ is short and $\alpha + \gamma$ is longer: it follows that
$(\alpha + \gamma)^\vee = (\alpha + \gamma)/2$ and $\alpha^\vee = 2\alpha$.
Again from (ii) we get that 
$$
(\alpha + \gamma)/2 = 2\alpha + \phi
\quad
\text{and thus}
\quad
3\alpha = 2\phi - \gamma
$$
As in item (b) we have that $\phi \in R_\T$ which shows that $3\alpha \in R_\T$.  By the above equation this implies that the root $\alpha$ lies in $R_\T$ and thus in $\Pi_\T$,  a contradiction.
\end{enumerate}

That (iii) implies (i) follows directly from Proposition \ref{propos1}.
\end{proof}

The next result characterizes full transitivity by duality.

\begin{corollary}
\label{corol3}
The following are equivalent:
\begin{enumerate}[(i)]
\item $W_\Theta$ is transitive in each nonzero residue classes mod $R_\T$ and mod $R_\T^\vee$
\item $\Pi_\T(\alpha)^\vee = \Pi^\vee_\T(\alpha^\vee)$ for all roots $\alpha$.
\item Either the root system is simply laced and $\Theta$ is arbitrary or else the root system is not simply laced and $\Theta = \emptyset$. 
\end{enumerate}
\end{corollary}
\begin{proof}
Item (i) implies, by the previous Lemma, that 
$\Pi_\T(\alpha)^\vee \subseteq \Pi^\vee_\T(\alpha^\vee)$ and
$\Pi^\vee_\T(\alpha^\vee)^\vee \subseteq \Pi_\T(\alpha)$, for all roots $\alpha$.
 Applying the duality map to the last equation we get 
 $\Pi^\vee_\T(\alpha^\vee)
  \subseteq \Pi_\T(\alpha)^\vee$ and thus get item (ii).
Again by the previous Lemma, item (ii) clearly implies item (i).

To prove that (i) implies (iii) use the long root criterion (Proposition \ref{propos2}) so that (i) implies that the roots in both $\Theta$ and $\Theta^\vee$ are long.
Since duality exchanges long and short roots this can only happen if $\Theta$ is empty or if the root system is simply laced and does not contain short roots. 

Item (iii) implies (i) by Remark \ref{remark:simply-laced}, when the root system is simply laced. When $\Theta = \emptyset$ then (i) follows trivially since in this case $W_\Theta = \{1\}$ and the classes mod $R_\T$ and $R^\vee_\T$ are singletons.
\end{proof}

\end{document}